\documentclass[10pt]{article} 
\usepackage{graphicx} 
\usepackage{amsmath, amssymb, amsthm}
\usepackage{mathrsfs}
\usepackage{latexsym, multicol, color, epsfig, multirow, url, amsbsy}
\usepackage{soul}
\usepackage[authoryear,longnamesfirst]{natbib}
\usepackage{setspace} 
\usepackage{algorithm}
 \usepackage{algpseudocode}
\usepackage{fullpage}
\usepackage[colorlinks=true, allcolors=blue]{hyperref} 

\newtheorem{theorem}{Theorem}
\newtheorem{lemma}[theorem]{Lemma}
\newtheorem{proposition}[theorem]{Proposition}
\newtheorem{definition}[theorem]{Definition}
\newtheorem{remark}[theorem]{Remark}

\usepackage[normalem]{ulem}

\title{Heavy Tailed Homogeneous Structural Causal Models}
\author{
Vishal Routh\\
v.routh@uga.edu\\
University of Georgia
\and
Shuyang Bai\\
bsy9142@uga.edu\\
University of Georgia
}
\date{\today}

\begin{document}

\maketitle

\begin{abstract}
We consider causal discovery in structural causal models driven by heavy-tailed noise, where extremes carry important information about causal direction. We introduce the Heavy-Tailed Homogeneous Structural Causal Model (HT-HSCM), a unified framework that generalizes heavy-tailed linear and max-linear models. We demonstrate that causal tail coefficients identify the complete ancestral partial order of the underlying directed acyclic graph. We also formulate a recursive algorithm for recovering quantities associated with the model called ancestral impulse-responses  from the causal tail coefficients. Our results provide a general and theoretically justified framework for causal discovery in heavy-tailed systems.
\end{abstract}

\section{Introduction}

Causal discovery---the process of learning the underlying causal structure from purely observational data---remains a fundamental challenge in modern statistics and machine learning. Traditionally, structural causal models (SCMs) have often relied heavily on assumptions  of finite-variance noise; for example, see \cite{chaudhuri2025consistent} and the references therein.  However, empirical data across various complex systems frequently exhibit heavy-tailed behavior, where extreme events play a disproportionate and highly informative role in the system's dynamics. Such considerations arise in fields including finance \cite{chuang2009causality}, Earth and environmental sciences \cite{sun2021causal}, \cite{mhalla2020causal}, public health \cite{chernozhukov2011inference}, genetics \cite{duncan2011genome}, and neuroscience \cite{zanin2016causality}, among others.

When the noise variables driving an SCM are heavy-tailed, the asymmetry in how extreme values propagate through the graph presents a unique opportunity. Recent literature has demonstrated that extreme value theory, particularly the framework of regular variation, can be leveraged to resolve causal directions. Existing works have successfully explored specific instances of this phenomenon---such as heavy-tailed linear SCMs (see \cite{gnecco2021causal}, \cite{pasche2023causal}, \cite{krali2025causal}, \cite{jiang2025separation}) and max-linear models (see \cite{kluppelberg2021estimating}, \cite{buck2021recursive}, \cite{gissibl2018tail}, \cite{tran2024estimating}, \cite{amendola2022conditional},  \cite{adams2025inference}). \cite{chavez2024causality} summarizes these developments in the review. However, there remains a critical need for a generalized framework capable of accommodating diverse, nonlinear functional relationships. We also mention the recent works of \cite{fang2025structural} and \cite{engelke2025extremes}, which formulate SCMs  at the level of the extremal limit.

In this paper, 
 
we introduce the \emph{Heavy-Tailed Homogeneous Structural Causal Model (HT-HSCM)}. We assume that the local structural equations to be non-negative , continuous, and 1-homogeneous structural functions, the HT-HSCM serves as a unifying framework that   encompasses   linear models,   max-linear models, and $\ell_p$-norm aggregation models under a single theoretical umbrella.

We show that the Causal Tail Coefficient (CTC)  originally introduced  in \cite{gnecco2021causal} for linear models continues to capture the asymmetric tail dependence between variables and inherently encodes the causal topology of the network for  HT-HSCMs.
Then, we formulate at the population level a recursive algorithm that leverages  the CTCs   to  recover certain   characteristics of the HT-HSCM that we term \emph{Ancestral Impulse-Responses (AIRs)}. Together, these results provide a robust, theoretically grounded pathway for performing causal discovery in   heavy-tailed environments.

\section{Preliminaries}

In this section, we prepare some   graph-theoretic notions as well as some  background on regular variation, the mathematical framework for describing heavy tails.

\subsection{Graph  notions}
Let $\mathcal{D}=(V,E)$ be a directed acyclic graph (DAG) with node set $V=\{1, \cdots, d\}$ and edge set $E \subset V \times V$. 
Node $j$ is a parent of node $i$ if the edge $(j,i) \in E$. We define the parent set of node $i$ as $pa(i)=\{j \in V: (j,i) \in E \}$.
We also write $j \to i$ if  $(j,i)\in E$. 
For an SCM defined on a DAG, each node in \(V\) corresponds to a random variable, and a directed edge indicates that the variable at the child node functionally depends on the variable at the parent node.

By a directed path from node $j$ to node $i$,  we mean a sequence of nodes $ (k_0 ,\cdots, k_s)$, $s\in \mathbb{Z}_+$,  such that $k_l \to k_{l+1}$, $ l=0,\cdots (s-1)$, and $k_0=j$, $k_s=i$.  If there exists a directed path from node $j$ to node $i$, we write $j \leadsto i$. If $j \leadsto i$, we say that $j$ is an ancestor of   $i$ and $i$ is a descendant of $j$. We denote the set of ancestors of node $i$ with $an(i)$ and $An(i)= an(i) \cup \{i\}$. Similarly we define the descendants of node $j$ with $de(j)$ and $De(j) = de(j) \cup \{j\}$.

 While the graph theoretic notations above provide the combinatorial structure of our causal models, the statistical identifiability of these structures relies on the specific tail behavior of the associated   random variables. In particular, to resolve causal directions from observational data, we must characterize how extreme values propagate through the network. This necessitates a formal treatment of regular variation, which provides the mathematical language to describe the heavy-tailed noise distributions that drive our proposed HT-HSCM. We detail these probabilistic preliminaries in the following subsection.

\subsection{Regularly varying functions and random variables}
Recall a positive measurable function $f$  is said to be \emph{regularly varying} (at $\infty$)  with index 
$\alpha \in \mathbb{R}$, denoted $f \in RV_\alpha$, if it is defined on some neighborhood of 
infinity $[x_0, \infty)$ for some $x_0 > 0$, and satisfies
\[
\lim_{x \to \infty} \frac{f(cx)}{f(x)} = c^{\alpha}, \quad \forall\, c > 0.
\]
The special case $\alpha = 0$ corresponds to \emph{slow variation}, written $f \in RV_0$.

A random variable $X$ is called \emph{regularly varying} with index $\alpha$, denoted as $RV_+(\alpha)$ if its tail distribution satisfies
\[
\mathbb{P}(X > x) \sim \ell(x) x^{-\alpha}, 
\] as $x \to \infty$.
for some slowly varying function $\ell \in RV_0$.  For two functions $f$ and $g$, the notation 
$f \sim g$ means that $f(x)/g(x) \to 1$ as $x \to \infty$.  
Classical examples of regularly varying random variables include those with Student’s-$t$, Pareto, Cauchy, and Fr\'echet distributions.
A key property of regularly varying random variables is the so-called \textit{single big jump principle}, which informally states that when an extreme event occurs, it is most likely driven by a single variable taking an exceptionally large value, while the remaining variables stay comparatively small. See the Supplementary Materials for a more detailed discussion.

\section{Heavy Tailed Homogeneous Structural Causal Models}

\subsection{Definitions}
Now we present the \emph{Heavy-Tailed Homogeneous Structural Causal Models} (HT-HSCMs), which encompass the max-linear and sum-linear heavy-tailed structural causal models previously considered in the literature.

\begin{definition}\label{Def:HT-HSCM}

 A HT-HSCM $\boldsymbol{X}=\{X_1,\cdots,X_d\}$, $d \in \mathbb{Z}_+$ on a DAG $\mathcal{D}=(V,E)$ with nodes $V=\{1,\cdots,d\}$ and edge set $E=\{(j,i): i \in V \text{ and } j \in pa(i)\}$ is a set of d assignments satisfying:
\begin{equation}
    X_i= f_{pa(i)}(\boldsymbol{X}_{pa(i)}, \epsilon_i), \quad i=1,\cdots,d
\end{equation}
where:
\begin{itemize}
    \item $\boldsymbol{X}_{pa(i)}=\{X_j: j \in pa(i)\}$ denotes the vector of parent variables of $X_i, \quad i=1,\cdots,d$. 
    \item $\varepsilon_1, \ldots, \varepsilon_d$ are i.i.d.\ regularly varying  random variables   with tail index $\alpha$, namely,
    \begin{equation}\label{eq:RV tail}
          \mathbb{P}(\varepsilon_i > x) \sim l(x)x^{-\alpha}, 
        \quad x \to \infty,
    \end{equation}
    for some slowly varying function $\ell$. 
    \item Each   structural function $f_{{pa}(i)} : \mathbb{R}_+^{|{pa}(i)|+1} \to \mathbb{R}_+$ satisfies:
    \begin{enumerate}
        \item \textbf{Non-negativity:} $f_{{pa}(i)}(x) \geq 0$ for all $x \in \mathbb{R}_+^{|{pa}(i)|+1}$.
        \item \textbf{Vanishing only at the origin:} $f_{{pa}(i)}(x) = 0 \iff x = 0$.
        \item \textbf{Continuity:} $f_{{pa}(i)}$ is continuous.
        \item \textbf{Homogeneity of degree 1:} for any $c > 0$ and $x \in \mathbb{R}_+^{|{pa}(i)|+1}$,
        \begin{equation}
            f_{{pa}(i)}(c x) = c f_{{pa}(i)}(x).
        \end{equation}
        \item \textbf{Coordinate deletion monotonicity}: Let $J \subseteq \{1, \dots, |pa(i)|+1\}$ be any nonempty subset of indices. The function $f_{pa(i)}$ satisfies:
\begin{equation}
    f_{pa(i)}(x_1, \dots, x_{|pa(i)|+1}) \ge f_{pa(i)}\left( \sum_{j \in J} x_j \mathbf{e}_j \right),
\end{equation}
where $\mathbf{e}_j$ is the $j$th standard orthonormal basis of $\mathbb{R}^{|{pa}(i)|+1}$.

    \end{enumerate}
\item We assume in addition that the marginal distribution of each $X_i$, $i=1,\ldots,d$, is continuous.
\end{itemize}
\end{definition}

\begin{remark}
Since $f_{pa(i)}$ is   homogeneous, when considering the specific case of Property 5 above where $x_j = c$ for all $j \in J$, $c\in (0,\infty)$, we obtain:
\begin{equation}\label{eq:monotone imp}
    f_{pa(i)}(\mathbf{x}) \ge f_{pa(i)}(c \mathbf{1}_J) = c f_{pa(i)}(\mathbf{1}_J),
\end{equation}
where $\mathbf{1}_J = \sum_{j \in J} \mathbf{e}_j$ is the indicator vector containing $1$s for all indices in $J$ and $0$s otherwise.
\end{remark}

By recursively substituting the local structural equations along the ancestral relations, each component $X_i$ of $\boldsymbol{X}$ can be expressed as a function $F_{{An}(i)}$ of its ancestral noise variables:
\begin{equation}\label{eq:anc agg}
    X_i = F_{{An}(i)} (\boldsymbol{\varepsilon}_{An(i)}), \quad i=1,\dots,d,
\end{equation}
where $\boldsymbol{\epsilon}_{An(i)}= (\epsilon_h)_{h \in An(i)}$ . It is straightforward to verify that each function $F_{{An}(i)} \colon \mathbb{R}_+^{|An(i)|} \to \mathbb{R}_+$, for $i=1, \dots, d$, inherits properties 1 through 5 from the local structural functions $f_{{pa}(i)}$.

\begin{definition}\label{Def:AIR}
    We define the Ancestral Impulse-Response  (AIR) matrix  $F \in \mathbb{R}^{d \times d}$ by  
\[
    F_{hi} = \begin{cases} 
        F_{{An}(i)}\!\left( (\mathbf{1}_{\{j=h\}})_{j \in {An}(i)} \right) & \text{if } h \in {An}(i), \\
        0 & \text{otherwise.}
    \end{cases}
\]
\end{definition}
If $h\in  An(i)$, the entry $F_{hi}$ may be interpreted as the influence of a unit impulse at node $i$ originating from its ancestor $h$. The significance of these quantities can be understood through the single big jump principle: $F_{hi}$ characterizes the effect of a single unit jump from the noise variable associated with ancestor $h$, while all other ancestral noise variables remain inactive. Due to the homogeneity of $F_{{An}(i)}$, it suffices to consider  a unit impulse.
 Indeed, the sequence $\bigl(\mathbf{1}_{\{j=h\}}\bigr)_{j\in{An}(i)}$ is the canonical basis vector corresponding to ancestor $h$.

 Consequently, $F_{hi}$ measures the marginal effect of $h$ on $i$ induced by $F_{{An}(i)}$. 
 Note that if $h \notin {An}(i)$, the input is the zero vector, which evaluates to exactly zero, ensuring that the matrix $F$ strictly respects the underlying ancestral structure. Thus, $F$ can be viewed as a structural influence matrix, with each column $i$ encoding how every other node individually contributes to node $i$.
For later use, it will be convenient to introduce the following variant of the AIR matrix $\widetilde{F} $, obtained by standardizing each column of the orignal AIR matrix $F$ with respect to the $\alpha$-norm.
\begin{definition}\label{Def:st AIR}
 We define the \emph{standardized AIR matrix} $\widetilde{F} \in \mathbb{R}^{d \times d}$ by  
\[
    \widetilde{F}_{hi} = \frac{F_{hi}}{\left(\sum_{k \in {An}(i)} F_{ki}^{\alpha}\right)^{1/\alpha}},
\]
where $F_{hi}$'s are as in Definition \ref{Def:AIR}. 
\end{definition}

\subsection{Examples}

In this section, we present several concrete examples of SCMs that fall within the class of HT-HSCMs introduced in Definition~\ref{Def:HT-HSCM}. Throughout, we assume that $\varepsilon_1,\ldots,\varepsilon_d$ are i.i.d.\ non-negative regularly varying random variables satisfying \eqref{eq:RV tail}.

\begin{enumerate}

\item \textbf{Linear SCM}.
Consider the structural equation
\[
    X_i = \sum_{h \in {pa}(i)} c_{hi} X_h + \varepsilon_i,
\]
where the structural function is given by
\[
    f_{{pa}(i)}(x_{{pa}(i)}, z) := \sum_{h \in {pa}(i)} c_{hi} x_h + z,
\]
 and the coefficients $c_{hi} \ge 0$.
 By recursively composing the structural functions along the DAG, we obtain an induced ancestral aggregation map $F_{{An}(i)}$. 
 In the linear SCM, this recursive substitution expands additively over the network topology. 
Let $\mathcal{P}_{hi}$ denote the set of all directed paths from node $h$ to node $i$. Then the AIRs in Definition \ref{Def:AIR} are given by
\begin{equation}\label{eq:linear SCM AIR}
    F_{hi} = 
    \begin{cases} 
        \sum_{\pi \in \mathcal{P}_{hi}} \prod_{(u, v) \in \pi} c_{uv}, & h \in {an}(i), \\[4pt]
        1, & h = i, \\[4pt]
        0, & h \notin {An}(i).
    \end{cases}
\end{equation}

Here, the product $\prod_{(u,v)\in \pi}$ is taken over all directed edges along the path $\pi$, where with a slight abuse of notation, we identify the path $\pi$ with the set of its directed edges.

 So that the matrix $F$ encodes the total marginal influence of each ancestor by summing the accumulated weights across all possible mediating paths.

\item \textbf{Max-Linear   Model}.
Consider the structural equation
\[
    X_i = \bigvee_{h \in {pa}(i)} (c_{hi} X_h) \vee \varepsilon_i,
\]
where the structural function is given by  
\[
    f_{{pa}(i)}(x_{{pa}(i)}, z) := \bigvee_{h \in {pa}(i)} (c_{hi} x_h) \vee z,
\]
and the coefficients $c_{hi} \ge 0$.
The AIRs in Definition \ref{Def:AIR} are given by
\[
    F_{hi} = \begin{cases} 
        \bigvee_{\pi \in \mathcal{P}_{hi}} \prod_{(u , v) \in \pi} c_{uv}, & h \in {an}(i), \\[4pt]
        1, & h = i, \\[4pt]
        0, & h \notin {An}(i). 
    \end{cases}
\]
 so that the matrix $F$ encodes the dominant marginal influence of each ancestor, isolating the single most heavily weighted path transmitting the extreme noise from $h$ to $i$.

\item \textbf{$\ell_p$  Model.}
Consider the structural equation
\begin{equation}
    X_i
    \;=\;
    f_{{pa}(i)}\!\Bigl( \bigl( (c_{hi} X_h)_{h\in{pa}(i)},\, \varepsilon_i \bigr) \Bigr),
     \qquad i=1,\cdots,d.
    \label{eq:htlp}
\end{equation}
where the structural function is given by the $p$-norm
\[
f_{{pa}(i)}(x_{{pa}(i)}, z)
\;:=\;
\Bigl\| \bigl( (c_{hi} x_h)_{h\in{pa}(i)},\, z \bigr) \Bigr\|_p, \qquad p\in(0,\infty)
\]
the coefficients $c_{hi}\ge 0$. Although the function $F_{{An}(i)}$ in \eqref{eq:anc agg} differs from that in the linear SCM, it can be verified that its AIRs coincide with those of the linear SCM given in \eqref{eq:linear SCM AIR}.

 By recursively composing the structural functions along the DAG, we obtain an induced ancestral aggregation map $F_{{An}(i)}$. 
 In the $\ell_p$-aggregation model~\eqref{eq:htlp},Letting $\mathcal{P}_{hi}$ denote the set of all directed paths from node $h$ to node $i$, this definition yields
 The AIRs in Definition \ref{Def:AIR} are given by
 \[
 F_{hi}
 =
 \begin{cases}
 \sum_{\pi \in \mathcal{P}_{hi}} \prod_{(u,v) \in \pi} c_{uv}, & h \in {an}(i), \\[4pt]
         1, & h = i, \\[4pt]
         0, & h \notin {An}(i),
 \end{cases}
 \]
which agrees with formula \eqref{eq:linear SCM AIR} in the linear case, once the coefficients \(c_{hi}\) in the two formulations are identified.
 % so that, despite the nonlinear aggregation at the parent level, the matrix $F$ continues to encode the marginal influence of each ancestor and respects the sparsity structure of the underlying DAG.

\end{enumerate}

We mention that more examples can be constructed by mixing different types of the  structural equations mentioned above at different nodes.

\section{Causal tail coefficient}
For any node $i \in V$, let $G_i$ denote the cumulative distribution function of random variable $X_i$.

\begin{definition}\label{Def:CTC}
    For any two random variables $X_i$ and $X_j$ in an HT-HSCM over $d$ variables as described in Definition \ref{Def:HT-HSCM}, we define the standardized Causal Tail Coefficient (CTC)  as:
    \begin{equation}\label{eq:st CTC}
        \Gamma^*_{ji} = \lim_{x \to \infty} \mathbb{E}[2G_i(X_i) - 1 \mid X_j > x],
    \end{equation}
    where $i, j \in \{1, \dots, d\}$ (the existence of these limits will be justified in Lemma \ref{Lem:CTC expr} below).   We also define  the standardized CTC  matrix $\boldsymbol{\Gamma}^* \in \mathbb{R}^{d \times d}$ as the matrix whose $(i,j)$-th entry is given by:
    \[
        (\boldsymbol{\Gamma}^*)_{ij} = \Gamma^*_{ij} \quad \text{for } i,j = 1, \dots, d.
    \]
\end{definition}

The standardized CTC  in \eqref{eq:st CTC} is simply an affine transformation of the CTC 
\[
\Gamma_{ji}= \lim_{x \to \infty} \mathbb{E}\!\left[ G_i(X_i) \mid X_j > x \right]= \lim_{x \to \infty} \mathbb{E}\!\left[ G_i(X_i) \mid G_j (X_j) > x \right]
\]
introduced in \cite{gnecco2021causal}. Note that, whenever it exists, the standardized CTC takes values in the interval $[0,1]$, in contrast to the range $[1/2,1]$ of the original CTC. 
Its introduction simplifies the presentation of the subsequent results.

The following lemma establishes the existence of the limits in \eqref{eq:st CTC} and relates them to the standardized AIRs defined in Definition~\ref{Def:st AIR}.

\begin{lemma} \label{Lem:CTC expr}
    For any two distinct nodes $i, j \in V$, we have
    \begin{equation}
        \Gamma^*_{ji} = \sum_{h \in {An}(i) \cap {An}(j)} \widetilde{F}_{hj}^\alpha
        \end{equation}
\end{lemma}
See \textit{Supplementary Materials} for a proof.\\
The following theorem demonstrates that the causal tail coefficient,  which is observable from the bivariate distribution of $X_1$ and $X_2$, fully encodes  the causal relationship between the two variables.

\begin{theorem}\label{Thm:CTC}
    Consider a HT-HSCM over $d$ variables, including $X_1$ and $X_2$, as defined in Section 1. Then, knowledge of $\Gamma^*_{12}$ and $\Gamma^*_{21}$ allows us to distinguish the following cases:
    \begin{enumerate}
        \item[(a)] $X_1$ causes $X_2$,
        \item[(b)] $X_2$ causes $X_1$,
        \item[(c)] There is no causal link between $X_1$ and $X_2$ (i.e., $An(1) \cap An(2) = \emptyset$),
        \item[(d)] There is a node $j \notin \{1, 2\}$ such that $X_j$ is a common cause of $X_1$ and $X_2$, and neither $X_1$ causes $X_2$ nor $X_2$ causes $X_1$.
    \end{enumerate}
\end{theorem}
The corresponding values for $\Gamma_{12}^*$ and $\Gamma_{21}^*$ are depicted in Table \ref{tab:causal_cases}.

\begin{table}[htpb]
    \centering
    \renewcommand{\arraystretch}{1.5}
    \caption{Summary of the possible values of $\Gamma^*_{12}$ and $\Gamma_{21}$ and the implications for causality.}
    \label{tab:causal_cases}
    \begin{tabular}{lccc}
        \hline
        & $\Gamma^*_{21} = 1$ & $\Gamma^*_{21} \in (0, 1)$ & $\Gamma^*_{21} = 0$ \\
        \hline
        $\Gamma^*_{12} = 1$          & --- & (a) $X_1$ causes $X_2$ & --- \\
        $\Gamma^*_{12} \in (0, 1)$ & (b) $X_2$ causes $X_1$ & (d) Common cause       & --- \\
        $\Gamma^*_{12} = 0$        & --- & --- & (c) No causal link \\
        \hline
    \end{tabular}
\end{table}

\begin{proof}
The proof of this theorem is based on Lemma~\ref{Lem:CTC expr} and closely follows that of Theorem~1 in \cite{gnecco2021causal}. We therefore omit the details.
\end{proof}

\begin{remark}
    
Theorem~\ref{Thm:CTC} can be interpreted at a more structural level. 
Indeed, cases~(a) and~(b), together with their complement, determine for any pair $(i,j)$ whether $i \in an(j)$, $j \in an(i)$, or neither holds. Hence the theorem identifies the partial order on the node set induced by the ancestor relation of the DAG. As a consequence, we can define a topological layering of nodes based  called generations (see \cite{kluppelberg2021estimating} Definition 1 for a discussion on generation of nodes and Lemma 1 in \cite{zhou2024efficient} on how to utilize the causal tail coefficients to find the generation of nodes). We may also apply the \textit{EASE} algorithm as in \cite{gnecco2021causal} to obtain a causal order of the nodes.

However, the content of Theorem~\ref{Thm:CTC} is strictly stronger than recovery of the partial order alone. Indeed, among pairs $(i,j)$ for which neither $i \in an(j)$ nor $j \in an(i)$ holds, the theorem further distinguishes between the case $an(i)\cap an(j)=\emptyset$ and the case where $i$ and $j$ share a common ancestor. This additional distinction is not encoded in the partial order itself.
\end{remark}

\section{Identifiability of AIRs using  the causal tail coefficients}

Throughout this section, we assume that the standardized CTC matrix  $\boldsymbol{\Gamma}^* = (\Gamma^*_{ij})_{i,j=1}^d$ in  Definition \ref{Def:CTC} is given. We present a population-level algorithm in Algorithm \ref{alg:findF}  that shows the standardized AIR matrix $\widetilde{F} = (\widetilde{F}_{ij})_{i,j=1}^d$  in Definition \ref{Def:AIR} can be identified from the standardized CTC  matrix $\boldsymbol{\Gamma}^*$ in Definition \ref{Def:CTC}. The algorithm is similar in spirit  
 to  Algorithm 4.1 in \cite{gissibl2018tail}, where tail dependence coefficient instead of CTC is involved in that work.

% For an HT-HSCM over $d$ vertices, The following algorithm, directly identifies the standardized AIR matrix $\widetilde{F}$ in Definition \ref{Def:st AIR}    from $\boldsymbol{\Gamma}^*$.

 \begin{algorithm}[h]
 \caption{Find $\widetilde{F}$ from $\boldsymbol{\Gamma}^*$}
\label{alg:findF}
 \begin{algorithmic}[1]
 \For{$\nu = 0, \dots, d-1$}
     \For{each $j \in V$ with $|{an}(j)| = \nu$}
         \For{each $i \in V \setminus {De}(j)$}
             \State $\widetilde{F}_{ji} \gets 0$
         \EndFor
         \For{each $i \in {De}(j)$}
             \State $\widetilde{F}_{ji} \gets \Gamma^*_{ij} - \sum_{k \in {an}(j)} \widetilde{F}_{ki}$
         \EndFor
     \EndFor
 \EndFor
 \end{algorithmic}
 \end{algorithm}

Note that  Algorithm \ref{alg:findF} requires knowledge of the cardinalities $|an(j)|$ for all $j \in V$. The following lemma shows that the ancestor set of each node can be identified from the standardized CTCs $\Gamma^*_{ij}$. 

% We then use this characterization to prove the correctness of the algorithm.

\begin{lemma}\label{lem:ancestor_set_gamma}
For every $j \in V$,
\begin{equation}
an(j)= \{\, i \in V \setminus \{j\} : \Gamma^*_{ij}=1 \,\}.
\end{equation}
In particular,
\begin{equation}
|an(j)| = \sum_{i \in V \setminus \{j\}} \mathbf{1}\{\Gamma^*_{ij}=1\}.
\end{equation}
\end{lemma}

\begin{proof}
By Theorem~\ref{Thm:CTC}, for any distinct $i,j \in V$, we have $\Gamma^*_{ij}=1$ if and only if $i \in an(j)$. The claimed set identity follows immediately, and the formula for $|an(j)|$ is then obtained by taking cardinalities.
\end{proof}

 Below, we formally establishes the correctness of Algorithm \ref{alg:findF}. 
\begin{proposition}\label{Pro:alg}
Algorithm \ref{alg:findF} correctly recovers the standardized   AIR matrix $\widetilde{F}$    from the standardized CTC  matrix $\boldsymbol{\Gamma}^*$.
\end{proposition}
\begin{proof} 
    We proceed to prove the correctness of this algorithm by induction on the size of the ancestor set, $|{an}(j)|$ which is determined by  Lemma~\ref{lem:ancestor_set_gamma}
    
    $\bullet$  Base case. 
    
    Suppose $j \in V$ such that $|{an}(j)| = 0$ (i.e., $j$ is a root node). For $i \in V \setminus {De}(j)$, we trivially have $\widetilde{F}_{ji} = 0$. For $i \in {De}(j)$, since $j$ is an ancestor of $i$ and has no ancestors itself, the intersection simplifies to ${An}(i) \cap {An}(j) = {An}(j) = \{j\}$. Therefore, evaluating the causal tail coefficient yields:
    \[
        \Gamma^*_{ij} = \sum_{h \in {An}(i) \cap {An}(j)} \widetilde{F}_{hi} = \widetilde{F}_{ji}.
    \]
    Thus, $\widetilde{F}_{ji}$ is fully identified for all root nodes.
    
     $\bullet$   Inductive step.
     
     Assume that for all nodes $j \in V$ with $|{an}(j)| \le \nu$, the entries $\widetilde{F}_{ji}$ have been correctly obtained. Consider a node $j' \in V$ with $|{an}(j')| = \nu + 1$. 
    
    For $i \in V \setminus {De}(j')$, it trivially holds that $\widetilde{F}_{j'i} = 0$. For $i \in {De}(j')$, the intersection of ancestors is ${An}(i) \cap {An}(j') = {An}(j')$. Expanding the causal tail coefficient gives:
    \[
        \Gamma^*_{ij} = \sum_{h \in {An}(i) \cap {An}(j')} \widetilde{F}_{hi} = \sum_{h \in {An}(j')} \widetilde{F}_{hi} = \widetilde{F}_{j'i} + \sum_{h \in {an}(j')} \widetilde{F}_{hi}.
    \]
    Observe that if $h \in {an}(j')$, then $h$ is a strict ancestor of $j'$, which  implies $|{an}(h)| \le \nu$ due to the acyclic nature of the graph. By our inductive hypothesis, the values $\widetilde{F}_{hi}$ in the summation are already known. Therefore, $\widetilde{F}_{j'i}$ is uniquely and deterministically identified. This concludes the proof.

\end{proof}

\section{Conclusion}

In this paper, we discussed causal discovery at the population level for structural causal models with heavy-tailed noise. Building on the causal tail coefficient framework of \cite{gnecco2021causal}, we extended these ideas from heavy-tailed linear models to a broader class of nonlinear systems through the introduction of the \emph{Heavy-Tailed Homogeneous Structural Causal Model} (HT-HSCM). This framework unifies several existing heavy-tailed causal models, including linear and max-linear models, under a common set of assumptions based on nonnegative, continuous, 1-homogeneous structural functions.

Our main results show that causal tail coefficients continue to encode fundamental structural information in this more general setting. In particular,  we presented a recursive procedure for recovering the standardized ancestral impulse-response matrix $\widetilde{F}$. These results provide a theoretical extension of the population-level identifiability results of \cite{gnecco2021causal} to a substantially richer class of heavy-tailed structural causal models.

Several important directions remain for future work. A natural next step is to utilize statistically consistent estimators of the standardized CTC matrix under finite-sample conditions, in the spirit of the nonparametric estimation approach proposed in \cite{gnecco2021causal} to show that the \textit{EASE} algorithm can be used to consistently recover causal orderings in the HT-HSCM setting.

 Another potential direction is to relax the nonnegativity constraint in HT-HSCM, thereby enabling the model to capture two-sided extremes.

\nopagebreak

\bibliographystyle{apalike}
\bibliography{References}

\newpage

\section*{Supplementary Material}

In the supplementary material, we present the proof of Lemma \ref{Lem:CTC expr}.

We start with the preparation of a few lemmas.
 
\begin{lemma}\label{Lem:uni RV}
    
A non-negative random variable $\epsilon$ is    regularly varying  with index $\alpha > 0$ if and only if   as $x \to \infty$,
$$
    \frac{\mathbb{P}(x^{-1}\epsilon \in \cdot)}{\mathbb{P}(\epsilon > x)} \xrightarrow{v} \nu_{\alpha}(\cdot)
$$
where $\nu_\alpha$ is a Borel measure on $(0,\infty)$ given by  $\nu_{\alpha}(A) = \int_A \alpha t^{-\alpha -1} \, dt$ for any Borel set $A \subset (0, \infty)$, and $\xrightarrow{v}$ denotes vague convergence with respect to the boundedness consisting of Borel subsets of $(0, \infty)$ that is each separated from the origin  (see \cite[Appendix B]{kulik2020heavy}).

\end{lemma}
\begin{proof}
    See  equation (1.3.6), (1.3.7), (1.3.8) and Theorem 2.1.3 (ii) in \cite{kulik2020heavy}.
\end{proof}

    \begin{lemma}\label{Lem:iid tail}
    Let $\boldsymbol{\epsilon} = (\epsilon_1, \dots, \epsilon_d)^\top$ be a vector of i.i.d.\ positive regularly varying random variables as in Lemma \ref{Lem:uni RV}. Then, as $x \to \infty$,
\[
    \frac{\mathbb{P}(x^{-1}\boldsymbol{\epsilon} \in \cdot)}{\mathbb{P}(\epsilon_1 > x)} \xrightarrow{v} \nu_{\boldsymbol{\epsilon}}(\cdot) = \sum_{i=1}^d \delta_0 \times \dots \times \nu_{\alpha} \times \dots \times \delta_0 (\cdot),
\]
where the vague convergence $\xrightarrow{v}$ is with respect to the boundedness consisting of Borel subsets of $[0,\infty)^d\setminus\{\mathbf{0}\}$ that is each separated from the the origin $\mathbf{0}=(0,\ldots,0)\in \mathbb{R}^d$.
\end{lemma}
\begin{proof}
    See Proposition 2.1.8 in \cite{kulik2020heavy}.
\end{proof}
 
\begin{lemma}\label{Lem:HTHSCM tail}
Consider a HT-HSCM $\boldsymbol{X}=\{X_1,\cdots,X_d\}$ as in Definition \ref{Def:HT-HSCM}.
  Then,
$$
    \lim_{x \to \infty} \frac{\mathbb{P}(X_1 > x)}{\mathbb{P}(\epsilon_1 > x)} =   \lim_{x \to \infty} \frac{\mathbb{P}\left( \bigvee_{h \in An(1)} F_{h1} \epsilon_h > x \right)}{\mathbb{P}(\epsilon_1 > x)} = \sum_{h \in An(1)} F_{h1}^\alpha,
$$
where $F_{h1}$'s are AIRs in Definition \ref{Def:AIR}.
\end{lemma}
\begin{proof}
   Our proof is inspired from Proposition 1 in \cite{fang2025structural}.
    
    By the 1-homogeneity of $F_{{An}(1)}$ in Definition \ref{eq:anc agg},  we have
    \[
        \mathbb{P}(X_1 > x) = \mathbb{P}\left(F_{{An}(1)}(\boldsymbol{\epsilon}_{An(1)}) > x\right) = \mathbb{P}\left(x^{-1}\boldsymbol{\epsilon}_{An(1)} \in F_{{An}(1)}^{-1} (1, \infty)\right).
    \]
    
    Without loss of generality, assume $\boldsymbol{\epsilon}_{An(1)} = (\epsilon_1, \dots, \epsilon_{|An(1)|})^\top$. Recall that for a continuous map, the boundary of the preimage of a set is contained in the preimage of its boundary. Hence,
\[
\partial F_{{An}(1)}^{-1}(1,\infty)
\subseteq
F_{{An}(1)}^{-1}\partial(1,\infty)
=
F_{{An}(1)}^{-1}\{1\}.
\] 
Evaluating the limit measure $ \nu_{\boldsymbol{\epsilon}_{An(1)}}= \sum_{i=1}^{|An(1)|} \delta_0 \times \dots \times \nu_{\alpha} \times \dots \times \delta_0  $ in Lemma \ref{Lem:iid tail} on this set  and applying Definition \ref{Def:AIR}  yields
    \[
    \nu_{\boldsymbol{\epsilon}_{An(1)}} \left( F_{{An}(1)}^{-1}  \{1\}   \right)    = \sum_{i=1}^{|An(1)|} \int_0^\infty \mathbf{1}_{\{tF_{i1}=1\}} \alpha t^{-\alpha-1} \, dt = 0.
    \]
     Thus, $\nu_{\boldsymbol{\epsilon}_{An(1)}}(\partial F_{{An}(1)}^{-1}((1, \infty))) = 0$. By the vague convergence established in Lemma \ref{Lem:iid tail}, we have as $x \to \infty$,
    \[
        \frac{\mathbb{P}(X_1 > x)}{\mathbb{P}(\epsilon_1 > x)} \to \nu_{\boldsymbol{\epsilon}_{An(1)}}\left(F_{{An}(1)}^{-1}(1, \infty)\right) = \sum_{i=1}^{|An(1)|} \int_0^\infty \mathbf{1}_{\{tF_{i1}>1\}} \alpha t^{-\alpha-1} \, dt = \sum_{i=1}^{|An(1)|} F_{i1}^\alpha.
    \]
 
    Furthermore,   note that by the inclusion-exclusion principle:
    \begin{align*}
        \mathbb{P}\left(\bigvee_{i=1}^{|An(1)|} F_{i1} \epsilon_i > x\right) &= \sum_{i=1}^{|An(1)|} \mathbb{P}(F_{i1} \epsilon_i > x) \\
        &\quad + \sum_{k \ge 2} \sum_{1 \le i_1 < i_2 < \dots < i_k \le |An(1)|} (-1)^{k+1} \mathbb{P}\left( \bigcap_{j=1}^k \{F_{i_j} \epsilon_{i_j} > x\} \right).
    \end{align*}
    Since the $\epsilon_{i}$'s are independent regularly varying random variables, $\mathbb{P}\left(\bigcap_{j=1}^k \{F_{i_j} \epsilon_{i_j} > x\}\right) = o(\mathbb{P}(\epsilon_1 > x))$ for $k \ge 2$. Finally, by regular variation,
    \[
        \frac{\mathbb{P}(F_{i1}\epsilon_i > x)}{\mathbb{P}(\epsilon_1 > x)} \to F_{i1}^\alpha
    \]
    for $1 \le i \le |An(1)|$. This matches the limit above and concludes our proof.
\end{proof}

    \begin{lemma}\label{Lem:exc and non-exc neg}
    Consider a HT-HSCM $\boldsymbol{X}=(X_i)_{i=1}^d$ as in Definition \ref{Def:HT-HSCM}.
Then, as $x \to \infty$,
\[
    \mathbb{P}\left(X_i > x, \bigvee_{h \in An(i)} F_{hi} \epsilon_h \le x\right) = o(\mathbb{P}(\epsilon_i > x)).
\] for all $i =1, \cdots ,d $
\end{lemma}
\begin{proof}
   Without loss  of generality assume $i=1$. We first note that relation \eqref{eq:monotone imp} and Definition \ref{Def:AIR} imply
    \begin{equation}\label{eq:single exc imp exc}
        \left\{ \bigvee_{h \in An(1)} F_{h1} \epsilon_h > x \right\} \subseteq \{X_1 > x\}.
    \end{equation}
    Therefore, we can decompose the joint probability as follows:
    \begin{align*}
        \mathbb{P}\left( X_1 > x, \bigvee_{h \in An(1)}F_{h1} \epsilon_h \le x \right) &= \mathbb{P}(X_1 > x) - \mathbb{P}\left( X_1 > x, \bigvee_{h \in An(1)}F_{h1} \epsilon_h > x \right) \\
        &= \mathbb{P}(X_1 > x) - \mathbb{P}\left( \bigvee_{h \in An(1)}F_{h1} \epsilon_h > x \right).
    \end{align*}
    Dividing the above relation by $\mathbb{P}(\epsilon_1 > x)$, we get:
    \[
        \frac{\mathbb{P}\left( X_1 > x, \bigvee_{h \in An(1)}F_{h1} \epsilon_h \le x \right)}{\mathbb{P}(\epsilon_1 > x)} = \frac{\mathbb{P}(X_1 > x)}{\mathbb{P}(\epsilon_1 > x)} - \frac{\mathbb{P}\left( \bigvee_{h \in An(1)}F_{h1} \epsilon_h > x \right)}{\mathbb{P}(\epsilon_1 > x)}.
    \]
    By Lemma \ref{Lem:HTHSCM tail}, both fractions on the right-hand side converge to the same limit as $x \to \infty$. Thus, their difference converges to $0$, which by definition means the joint probability is $o(\mathbb{P}(\epsilon_1 > x))$.
\end{proof}

We are now ready to prove Lemma \ref{Lem:CTC expr}.
\begin{proof}[Proof of Lemma \ref{Lem:CTC expr}]
   Our strategy of proof is similar to that of the proof of Lemma 1 in \cite{gnecco2021causal}. We begin by noting the conditional expectation can be written as:
    \[
        \mathbb{E}[G_2(X_2) \mid X_1 > x] = \frac{\mathbb{E}[G_2(X_2) \mathbf{1}_{\{X_1 > x\}}]}{\mathbb{P}(X_1 > x)}
    \]
    In view of the relation \eqref{eq:single exc imp exc}, we can decompose the indicator function as:
    \[
        \mathbf{1}_{\{X_1 > x\}} = \mathbf{1}_{\left\{ \bigvee_{h \in An(1)} F_{h1} \epsilon_h > x \right\}} + \mathbf{1}_{\left\{ X_1 > x, \, \bigvee_{h \in An(1)} F_{h1} \epsilon_h \le x \right\}}.
    \]
    Since $G_2$ is bounded between $0$ and $1$, we  have by Lemma \ref{Lem:exc and non-exc neg} that
    \[
        0 \le \mathbb{E}\left[ G_2(X_2) \mathbf{1}_{\left\{ X_1 > x, \, \bigvee_{h \in An(1)} F_{h1} \epsilon_h \le x \right\}} \right] \le \mathbb{P}\left( X_1 > x, \bigvee_{h \in An(1)} F_{h1} \epsilon_h \le x \right) = o(\mathbb{P}(\epsilon_1 > x)) \text{ as }x\rightarrow\infty.
    \]
    Therefore,
    \[
        \mathbb{E}[G_2(X_2) \mathbf{1}_{\{X_1 > x\}}] = \mathbb{E}\left[ G_2(X_2) \mathbf{1}_{\left\{ \bigvee_{h \in An(1)} F_{h1} \epsilon_h > x \right\}} \right] + o(\mathbb{P}(\epsilon_1 > x)).
    \]
    Now for the first term above, applying the inclusion-exclusion principle yields:
\begin{align*}
    \mathbb{E}\left[ G_2(X_2) \mathbf{1}_{\left\{ \bigvee_{h \in An(1)} F_{h1} \epsilon_h > x \right\}} \right] &= \sum_{h \in An(1)} \mathbb{E}[G_2(X_2) \mathbf{1}_{\{F_{h1} \epsilon_h > x\}}] \\
    &\quad + \sum_{k \ge 2} \sum_{1 \le h_1 < \dots < h_k} (-1)^{k+1} \mathbb{E}\left[ G_2(X_2) \mathbf{1}_{\left\{ \bigcap_{j=1}^k F_{h_j} \epsilon_{h_j} > x \right\}} \right].
\end{align*}
    Again  since $G_2$ is bounded by $1$,  and  $\epsilon_i$'s   are independent regularly varying random variables, the second summation  is of order $o(\mathbb{P}(\epsilon_1 > x))$ as $x\rightarrow\infty$. We now split the first summation above  into:
    \begin{equation}\label{eq:E[G_2 1] key relation}
        \sum_{h \in An(1)} \mathbb{E}[G_2(X_2) \mathbf{1}_{\{F_{h1} \epsilon_h > x\}}] = \sum_{h \in An(1) \cap An(2)} \mathbb{E}[G_2(X_2) \mathbf{1}_{\{F_{h1} \epsilon_h > x\}}] + \sum_{h \in An(1) \setminus An(2)} \mathbb{E}[G_2(X_2) \mathbf{1}_{\{F_{h1} \epsilon_h > x\}}].
    \end{equation}
    
    Note that for $h \in An(1) \setminus An(2)$, we have $\epsilon_h \perp X_2$, and that $\mathbb{E}[G_2(X_2)] = 1/2$ since $G_2$ has been assumed to be continuous and thus $G_2(X_2)$ follows a uniform distribution on $[0,1]$. Therefore,
    \[
        \mathbb{E}[G_2(X_2) \mathbf{1}_{\{F_{h1} \epsilon_h > x\}}] = \frac{1}{2} \mathbb{P}(F_{h1} \epsilon_h > x).
    \]
    So dividing   this by $\mathbb{P}(X_1 > x)$, and applying regular variation of $\varepsilon_h$ and Lemma \ref{Lem:HTHSCM tail}, we get as $x\rightarrow\infty$:
    \begin{equation}\label{eq:E[G_2 1] key relation 1}
        \frac{\mathbb{E}[G_2(X_2) \mathbf{1}_{\{F_{h1} \epsilon_h > x\}}]}{\mathbb{P}(X_1 > x)}  =  \frac{1}{2}\frac{ \mathbb{P}(F_{h1} \epsilon_h > x)] / \mathbb{P}(\epsilon_1 > x)}{\mathbb{P}(X_1 > x) / \mathbb{P}(\epsilon_1 > x)} \to \frac{1}{2} \frac{F_{h1}^\alpha}{\sum_{k \in An(1)} F_{k1}^\alpha}.
    \end{equation}

    Now we turn to the first summation in \eqref{eq:E[G_2 1] key relation}. For $h \in An(1) \cap An(2)$, the relation \eqref{eq:monotone imp} guarantees that $X_2 \ge F_{h2} \epsilon_h$. Given $F_{h1} \epsilon_h > x$, it follows that $X_2 > F_{h2} \frac{x}{F_{h1}}$. Because $G_2$ is a monotonically increasing CDF that is bounded by $1$, we have for all $x\ge 0$ that
    \[
        G_2\left( F_{h2} \frac{x}{F_{h1}} \right) \mathbb{P}(F_{h1} \epsilon_h > x) \le \mathbb{E}[G_2(X_2) \mathbf{1}_{\{F_{h1} \epsilon_h > x\}}] \le \mathbb{P}(F_{h1} \epsilon_h > x).
    \]
    Dividing the  relations above by $\mathbb{P}(X_1 > x)$ and taking the limit as $x \to \infty$,  applying also the fact that $G_2(c) \to 1$ as $c \to \infty$ and Lemma \ref{Lem:HTHSCM tail}, we then have  as $x\rightarrow\infty$,
    \begin{equation}\label{eq:E[G_2 1] key relation 2}
        \frac{\mathbb{E}[G_2(X_2) \mathbf{1}_{\{F_{h1} \epsilon_h > x\}}]}{\mathbb{P}(X_1 > x)} \to \frac{F_{h1}^\alpha}{\sum_{k \in An(1)} F_{k1}^\alpha}.
    \end{equation}

Combining \eqref{eq:E[G_2 1] key relation}, \eqref{eq:E[G_2 1] key relation 1} and \eqref{eq:E[G_2 1] key relation 2}, we have
\begin{align*}
    \lim_{x \to \infty} \mathbb{E}[G_2(X_2) \mid X_1 > x] &= \sum_{h \in An(1) \cap An(2)} \frac{F_{h1}^\alpha}{\sum_{k \in An(1)} F_{k1}^\alpha} + \sum_{h \in An(1) \setminus An(2)} \frac{1}{2} \frac{F_{h1}^\alpha}{\sum_{k \in An(1)} F_{k1}^\alpha} \\
    &= \frac{1}{2} + \frac{1}{2} \frac{\sum_{h \in An(1) \cap An(2)} F_{h1}^\alpha}{\sum_{h \in An(1)} F_{h1}^\alpha}= \frac{1}{2}+\frac{1}{2}\sum_{h \in An(1) \cap An(2)}\widetilde{F}_{h1}.
\end{align*}
Therefore,  \[
\Gamma^*_{12}= \lim_{x \to \infty} \mathbb{E}(2G_2(X_2))-1 | X_1 >x) = \sum_{h \in An(1) \cap An(2)}\widetilde{F}_{h1}
\]
This concludes the proof.
\end{proof}

\end{document}